\documentclass{amsart}
\usepackage{times}
\usepackage{amsfonts}
\usepackage{amssymb}

\newtheorem{thm}{Theorem}
\newtheorem{lemma}[thm]{Lemma}
\newtheorem{cor}[thm]{Corollary}
\newtheorem{prop}[thm]{Proposition}
\newtheorem{otherth}{\bf Theorem}

\newtheorem{otherl}[otherth]{\bf Lemma}

\newcommand{\Cn}{\mathbb{C}^n}
\newcommand{\Bn}{\mathbb{B}_ n}
\newcommand{\Sn}{\mathbb{S}_ n}

\def\C{\mathbb{C}}

\def\a{\alpha}

\def\e{\varepsilon}

\def\p{\varphi}

\def\s{\sigma}

\def\p{\varphi}

\def\inb{\int_{\Bn}}

\begin{document}

\title[Weak factorization and Hankel operators]
{Weak factorization and Hankel forms for weighted Bergman spaces
on the unit ball}

\author[J. Pau]
{Jordi Pau}
\address{
Jordi Pau\\
Department de Matem\`atica Aplicada i Analisi,
Universitat de Barcelona,
08007 Barcelona, Catalonia,
Spain}
\email{jordi.pau@ub.edu}

\author[R. Zhao]
{Ruhan Zhao}
\address{
Ruhan Zhao\\
Department of Mathematics,
SUNY Brockport,
Brockport, NY 14420,
USA}
\email{rzhao@brockport.edu}

\subjclass[2010]{32A36, 47B35, 47B38}

\keywords{Weak factorization, Hankel operators, Bergman spaces}

\thanks{This work started when the second named author visited the University of Barcelona in 2013. He thanks the support given by the IMUB during his visit.  The first author was
 supported by DGICYT grant MTM$2011$-$27932$-$C02$-$01$
(MCyT/MEC) and the grant 2014SGR289 (Generalitat de Catalunya)}


\begin{abstract}
We establish weak factorizations for a weighted Bergman space
$A^p_{\a}$, with $1<p<\infty$,
into two weighted Bergman spaces
on the unit ball of $\C^n$.
To obtain this result, we characterize bounded Hankel forms
on weighted Bergman spaces on the unit ball of $\C^n$.
\end{abstract}


\maketitle

\section{Introduction}
\label{intro}
A classical theorem of Riesz asserts that any function in the Hardy space $H^p$ on the unit disk can be factored as $f=Bg$ with $\|f\|_{H^p}=\|g\|_{H^p}$, where $B$ is a Blaschke product and $g$ is an $H^p$-function with no zeros on the unit disk. An immediate consequence of that result is that any function in the Hardy space $H^p$ admits a ``strong" factorization $f=f_ 1 f_ 2$ with $f_ 1 \in H^{p_ 1}$, $f_ 2\in H^{p_ 2}$ and $\|f_ 1\|_{H^{p_1}}\cdot \|f_ 2\|_{H^{p_ 2}}=\|f\|_{H^p}$, for any $p_ 1$ and $p_ 2$ determined by the condition $1/p=1/p_ 1+1/p_ 2$.  In \cite{ho}, C. Horowitz  obtained
strong factorizations of functions in a weighted Bergman space on the unit disk
into functions of two weighted Bergman spaces with the same weight (again using Blaschke products). These strong factorization results are no longer possible to obtain \cite{Gow} in the setting of Hardy and Bergman spaces in the unit ball of the complex euclidian space $\Cn$ of dimension $n$ when $n\ge 2$, but it is still possible to obtain some ``weak" factorizations for functions in these spaces.

For two Banach spaces  of functions, $A$ and $B$, defined on the same domain,
the weakly factored space $A\odot B$ is defined as the completion
of finite sums
$$
f=\sum_{k}\p_k\psi_k, \qquad \{\p_k\}\subset A, \ \{\psi_k\}\subset B,
$$
with the following norm:
$$
\|f\|_{A\odot B}=\inf\left\{\sum_k\|\p_k\|_A\|\psi_k\|_B:\,
f=\sum_{k}\p_k\psi_k\right\}.
$$
When $0<p\le1$, weak type factorizations for the Hardy spaces $H^p$ and the weighted Bergman spaces
$A^p_{\alpha}$ on the unit ball of $\C^n$ are well known
(see \cite{crw} and \cite{GL} for Hardy spaces; and \cite{cr}, \cite{Roc} or \cite[Corollary 2.33]{zhu2} for Bergman spaces).
However, when $1<p<\infty$, even for unweighted
Bergman spaces the problem was still open (see, for example \cite{bl}).

In this paper we completely solve the above problem for Bergman spaces by
establishing weak factorizations for a weighted Bergman space
$A^q_{\beta}$, with $1<q<\infty$ and $\beta>-1$,
into two weighted Bergman spaces with non necessarily the same weight,
on the unit ball $\Bn$ of $\C^n$. The following is our main result.

\begin{thm}\label{weak}
Let $1<q<\infty$ and $\beta>-1$.
Then
$$
A^q_{\beta}(\Bn)=A^{p_1}_{\a_1}(\Bn)\odot A^{p_2}_{\a_2}(\Bn)
$$
for any
$p_1,p_2>0$ and $\a_1,\a_2>-1$ satisfying
\begin{equation}\label{conjugate}
\frac{1}{p_1}+\frac{1}{p_2}=\frac{1}{q},\qquad
\frac{\a_1}{p_1}+\frac{\a_2}{p_2}=\frac{\beta}{q}.
\end{equation}
\end{thm}

In this context, by ``$=$" we mean equality of the function spaces
and equivalence of the norms. The inclusion $A^{p_1}_{\a_1}\odot A^{p_2}_{\a_2}\subset A^q_{\beta}$ with the estimate $\|f\|_{q,\beta}\lesssim \|f\|_{A^{p_1}_{\a_1}\odot A^{p_2}_{\a_2}}$ is a direct consequence of Minkowski and H\"{o}lder inequalities, so that the interesting part is the other inclusion with the corresponding estimates for the norms.

Now we are going to recall the definition of the weighted Bergman spaces. First we need some notations. For any two points $z=(z_1,...,z_n)$ and $w=(w_1,...,w_n)$ in $\C^n$,
we use
$$
\langle z,w\rangle=z_1\bar w_1+\cdots+z_n\bar w_n
$$
to denote the inner product of $z$ and $w$, and
$$
|z|=\sqrt{\langle z,z\rangle}=\sqrt{|z_1|^2+\cdots+|z_n|^2}
$$
to denote the norm of $z$ in $\C^n$.
Let $\Bn=\{z\in \C^n:|z|<1\}$ be the unit ball in $\C^n$
and $\Sn =\{z\in \C^n: |z|=1\}$ be the unit sphere in $\C^n$.
Let $H(\Bn)$ be the space of all analytic functions on $\Bn$.
We use $dv$ to denote the normalized volume measure on $\Bn$
and $d\s$ to denote the normalized area measure on $\Sn$.
For $-1<\a<\infty$, we let
$dv_{\a}(z)=c_{\a}(1-|z|^2)^{\a}\,dv(z)$ denote the normalized
weighted volume measure on $\Bn$,
where $c_{\a}=\Gamma(n+\a+1)/[n!\Gamma(\a+1)]$.

For $0<p<\infty$ and $-1<\a<\infty$, let
$L^{p}(\Bn,dv_{\a})$ be the weighted Lebesgue space
which contains measurable functions $f$ on $\Bn$
such that
$$
\|f\|_{p,\a}^p=\int_{\Bn}|f(z)|^p\,dv_{\a}(z)<\infty.
$$
Denote by $A^p_{\a}=L^{p}(\Bn,dv_{\a})\cap \,H(\Bn)$,
the weighted Bergman space on $\Bn$, with the same norm as above.
If $\a=0$, we simply write them as $L^p(\Bn,dv)$ and $A^p$ respectively
and $\|f\|_p$ for the norm of $f$ in these spaces.

It is a well-known fact that to obtain weak factorization results is equivalent
to give a ``good" description of the boundedness of certain \emph{Hankel forms}.
A Hankel form is a bilinear form $B$ on a space of analytic functions
such that for any $f$ and $g$, $B(f,g)$ is a linear function of $fg$.
These forms have been extensively studied on Hardy spaces and on Bergman spaces.
For the case of the Hardy space on the unit disk,
a classical result by Nehari \cite{ne} says that the Hankel form
$$
B_b(f,g):=\langle fg,b\rangle
$$
(under the usual integral pair for Hardy spaces)
with an \emph{analytic symbol} $b$ is bounded on $H^2\times H^2$
if and only if $b\in BMOA$, the space of analytic functions
of bounded mean oscillation.
The proof used the fact that a function in $H^1$
can be factored into product of two functions in $H^2$.
Unfortunately, such strong factorization is not possible (see \cite{Gow}) for Hardy spaces in the unit ball $\Bn$ of $\C^n$.
However, Coifman, Rochberg and Weiss \cite{crw} were able to generalize Nehari's
result to the unit ball $\Bn$ by using a weak factorization of $H^1$.
Namely, they proved that
$$
H^2(\Bn)\odot H^2(\Bn)=H^1(\Bn).
$$

Our approach to the problem for weighted Bergman spaces on the unit ball
is the opposite to the one of Coifman, Rochberg and Weiss in \cite{crw}.
We first characterize boundedness of the Hankel forms
on weighted Bergman spaces, and with this result
the weak factorization easily follows.

Given $\alpha>-1$ and a holomorphic symbol function $b$
we define the associated Hankel type bilinear form $T_ b^{\a}$ for polynomials $f$ and $g$ by
$$
T_b^{\a}(f,g)=\langle fg,b\rangle_{\a},
$$
where the integral pair $\langle \,,\,\rangle_{\a}$ is defined as
\begin{equation}\label{int-par}
\langle \varphi,\psi \rangle_{\a}=\inb \varphi(z)\,\overline{\psi(z)} \,dv_{\alpha}(z).
\end{equation}
Since the polynomials are dense in the weighted Bergman spaces, the Hankel form $T_ b^{\a}$ is densely defined on $A^{p_1}_{\a_1}\times A^{p_2}_{\a_2}$ for any $p_ 1,p_ 2>0$ and any $\a_ 1,\a_ 2>-1$.
We say that $T_ b^{\a}$ is bounded on $A^{p_ 1}_{\a_1}\times A^{p_ 2}_{\a_2}$ if there exists a positive constant $C$ such that
$$|T_ b^{\a}(f,g)|\le C \|f\|_{p_ 1,\alpha_ 1} \|g\|_{p_ 2,\alpha_ 2}.$$
The norm of $T_b^{\a}$ is given by
$$
\|T_b^{\a}\|
=\|T_b^{\a}\|_{A^{p_1}_{\a_1}\times A^{p_2}_{\a_2}}
:=\sup\{|T_b^{\a}(f,g)|:\,\|f\|_{p_1,\a_1}=\|g\|_{p_2,\a_2}=1\,\}.
$$
The next result characterizes boundedness of the Hankel form $T_b^{\a}$ acting on
 $A^{p_1}_{\a_1}\times A^{p_2}_{\a_2}$. We will see in Section \ref{hankel forms} that this implies the weak factorization in Theorem \ref{weak}.

\begin{thm}\label{hankel-form}
Let $1<p_1, p_2<\infty$, and
$\a, \a_1,\a_2>-1$ satisfy
\begin{equation}\label{pa-ineq}
\frac{1}{p_1}+\frac1{p_2}<1,\qquad
\frac{1+\a_1}{p_1}+\frac{1+\a_2}{p_2}<1+\a.
\end{equation}
Then $T_b^{\a}$ is bounded on $A^{p_1}_{\a_1}\times A^{p_2}_{\a_2}$
if and only
$b\in A^{q'}_{\beta'}$,
where $q$ and $\beta$ are real numbers satisfying \eqref{conjugate}, and
 $q'$ and $\beta'$ are determined by the condition
\begin{equation}\label{conjugate-q}
\frac{1}{q}+\frac1{q'}=1,\qquad
\frac{\beta}{q}+\frac{\beta'}{q'}=\a.
\end{equation}
Furthermore, we have
$
\|T_b^{\a}\|\asymp \|b\|_{q',\beta'}
$
\end{thm}

\textit{Remarks.}  Note that,
condition (\ref{pa-ineq}) guarantees that $q>1$ and $\beta'>-1$.
When $q$ and $\beta$ satisfy condition
(\ref{conjugate}), automatically we would have $\beta>-1$
(to see this, simply add two equations in (\ref{conjugate}) together). By a general duality theorem for weighted Bergman spaces (see Theorem \ref{duality} in Section \ref{pre}), the condition $b\in A^{q'}_{\beta'}$ means that the symbol $b$ belongs to the dual space of $A^q_{\beta}$ under the pairing given by \eqref{int-par}.\\

It turns out that boundedness of the Hankel form $T_b^{\a}$
is equivalent to boundedness of a (small) Hankel operator,
which we are going to introduce in a moment. Let $\a>-1$.
It is well-known that, the integral operator
$$
P_{\a}f(z)=\int_{\Bn}\frac{f(w)}{(1-\langle z,w\rangle)^{n+1+\a}}d\,v_{\a}(w)
$$
is the orthogonal projection from $L^2(\Bn,dv_{\a})$
onto the weighted Bergman space $A^2_{\a}$.
The above formula can be used to extend $P_{\a}$ to a linear operator
from $L^1(\Bn,dv_{\a})$ into $H(\Bn)$.
For $1<p<\infty$, $P_{\a}$ is a bounded operator
from $L^p(\Bn,dv_{\a})$ onto $A^p_{\a}$.

Denote by $\overline{A^p_{\a}}$ the conjugate
analytic functions $f$ on $\Bn$ that are in $L^p(\Bn,dv_{\a})$.
Clearly,
$$
\overline{A^p_{\a}}=\{\overline{f}:\,f\in A^p_{\a}\}.
$$
Let $Q_{\a}$ denote the orthogonal projection
from $L^2(\Bn,dv_{\a})$ onto $\overline{A^2_{\a}}$. Clearly one has
$$Q_ {\a} f(z)= \overline{P_{\a} \overline{f} (z)}=\int_{\Bn} \frac{f(w)}{(1-\langle w,z \rangle )^{n+1+\alpha} }\, dv_{\alpha}(w).$$
Given $f\in L^1(\Bn,dv_{\a})$ and a polynomial $g$,
the weighted (small) \emph{Hankel operator} is defined by
$$
h^{\a}_fg=Q_{\a}(fg).
$$
Due to the density of polynomials, the small Hankel operator $h^{\a}_f$ is densely defined on
the weighted Bergman space $A^p_{\a}$ for $1\le p<\infty$. We will study boundedness of the small Hankel operator with conjugate analytic symbols, that is, $h^{\a}_{\overline{f}}$ with $f\in H(\Bn)$,
from $A^{p_1}_{\a_1}$ to $\overline{A^{p_2}_{\a_2}}$ with $0<p_2<p_1<\infty$.

\begin{thm}\label{hankel}
Let $1<p_2<p_1<\infty$ and $\a_1,\a_2>-1$
such that
\begin{equation}\label{p-alpha}
\frac{1+\a_1}{p_1}<\frac{1+\a_2}{p_2}.
\end{equation}
Let $f\in H(\Bn)$ and $\a$ such that
\begin{equation}\label{alpha}
1+\a>\frac{1+\a_2}{p_2}.
\end{equation}
Then  $h^{\a}_{\bar{f}}:\,A^{p_1}_{\a_1}\to \overline{A^{p_2}_{\a_2}}$ is bounded if and only if
 $f\in A^q_{\beta}$, where $q$ and $\beta$ are real numbers such that
$$
\frac1q=\frac1{p_2}-\frac1{p_1},\qquad
\frac{\beta}q=\frac{\a_2}{p_2}-\frac{\a_1}{p_1}.
$$
Moreover,  we have $\|h^{\a}_f\|\asymp \|f\|_{q,\beta}$.
\end{thm}
\textit{Remarks.}
Condition (\ref{p-alpha}) guarantees that $\beta>-1$.
It is known that, when $0<p_2<p_1<\infty$,
$A^{p_1}_{\a_1}\subset A^{p_2}_{\a_2}$ if and only if
(\ref{p-alpha}) is true (see \cite[Theorem 70]{zz}). Hence the above result concerns the boundedness of $h^{\a}_f$ from a smaller space to a
larger space. Also, by \cite[Theorem 2.11]{zhu2}, condition \eqref{alpha} means that the integral operator $P_{\alpha}$ is a bounded projection from $L^{p_ 2}(\Bn,dv_{\alpha_ 2})$ onto $A^{p_ 2}_{\a_2}$.\\
\\
\medskip
If one considers the operator
$$
S^{\a}_fg=\overline{h_{\bar{f}}g(z)}=P_{\a}(f\overline{g}),
$$
clearly, the boundedness of $h^{\a}_{\overline{f}}$ is equivalent to the boundedness of $S^{\a}_f$
from $A^{p_1}_{\a_1}$ to $A^{p_2}_{\a_2}$,
and the norms of $h^{\a}_f$ and $S^{\a}_f$ are equivalent. Now, if $g\in A^{p_ 1}_{\a_ 1}$ and $h\in A^{p_2}_{\a_2}$,
by Fubini's theorem we easily obtain
$$
T_f^{\a}(g,h)=\langle gh, f\rangle_{\a}=\langle h, P_{\a}(f\bar g)\rangle_{\a}=\langle h, S^{\a}_fg\rangle_{\a}.
$$
Hence, for $p_ 2>1$, by duality (see Theorem \ref{duality} in Section \ref{pre}), the Hankel form
$T_f^{\a}$ is bounded on $A^{p_1}_{\a_1}\times A^{p_2}_{\a_2}$ if and only if the small Hankel operator
$h^{\a}_{\overline{f}}$ is bounded from $A^{p_1}_{\a_1}$ to $\overline{A^{p_2'}_{\a_2'}}$,
with equivalent norms. Here, the numbers $\alpha_ 2 '$ and $p'_ 2$ are defined by the relation
$$ \frac{1}{p_ 2}+\frac{1}{p'_ 2}=1,\qquad
\alpha=\frac{\a_ 2}{p_ 2}+\frac{\a'_ 2}{p_ 2'}.$$
Comparing Theorem~\ref{hankel-form} with Theorem \ref{hankel}, notice that the first inequality in (\ref{pa-ineq}) is equivalent to
condition $1<p_2'<p_1<\infty$. Also,
when $p_2$ and $\a_2$ are replaced by $p_2'$ and $\a_2'$,
condition (\ref{alpha}) turns out to be equivalent to $\a_ 2>-1$, and therefore is always satisfied; and the second inequality in (\ref{pa-ineq})
is equivalent to condition (\ref{p-alpha}). Therefore, Theorem \ref{hankel}
implies Theorem~\ref{hankel-form}.\\

In the case of the same weights, that is, when $\alpha_ 1=\alpha_ 2=\beta=\alpha$, all the restrictions in Theorem \ref{hankel} reduces to $p_ 2>1$. We isolate this case here, since it proves a conjecture in \cite{bl}.
\begin{cor}\label{C4}
Let $\alpha>-1$, $1<p_ 2<p_ 1<\infty$ and $f\in H(\Bn)$. Then $h^{\a}_ {\overline{f}}:A^{p_ 1}_{\a}\rightarrow \overline{A^{p_ 2}_{\a}}$ is bounded if and only if $f\in A^{q}_{\alpha}$, with $q=\frac{p_ 1 p_ 2}{p_ 1-p_ 2}$.
\end{cor}

The paper is organized as follows: in Section \ref{pre} we give some necessary concepts and recall some key results  which
are needed in our proof of the main result.
In Section \ref{hankel forms} we give in detail the connection between weak factorizations and Hankel forms.
The proof of Theorem \ref{hankel} is given in Section \ref{proof-thm1}.

In the following, the notation $A\lesssim B$ means
that there is a positive constant $C$ such that $A\leq CB$,
and the notation $A\asymp B$ means that both $A\lesssim B$ and $B\lesssim A$ hold.

\section{Preliminaries}
\label{pre}
We need the following duality theorem. In this generality
the result is due to Luecking \cite{l1}
(see also, Theorem 2.12 in \cite{zhu2}).

\begin{otherth}\label{duality}
Suppose $\beta,\beta'>-1$ and $1<q<\infty$. Then
$$
(A^{q}_{\beta})^*=A^{q'}_{\beta'}
$$
(with equivalent norms)
under the integral pair $\langle \,,\,\rangle_{\alpha}$ given by \eqref{int-par},
where
$$
\frac{1}{q}+\frac{1}{q'}=1,\qquad
\alpha=\frac{\beta}{q}+\frac{\beta'}{q'}.
$$
\end{otherth}

We need the following well known integral estimate that can be found, for example, in \cite[Theorem 1.12]{zhu2}.
\begin{otherl}\label{LA}
Let $t>-1$ and $s>0$. There is a positive constant $C$ such that
\[ \int_{\Bn} \frac{(1-|w|^2)^t\,dv(w)}{|1-\langle z,w\rangle |^{n+1+t+s}}\le C\,(1-|z|^2)^{-s}\]
for all $z\in \Bn$.
\end{otherl}

For any $a\in \Bn$ with $a\neq 0$, we denote by
$\p_a(z)$ the M\"obius transformation on $\Bn$
that exchanges $0$ and $a$. It is known that, for
any $z\in\Bn$
$$
\p_a(z)=\frac{a-P_a(z)-s_aQ_a(z)}{1-\langle z,a\rangle},
$$
where $s_a =1-|a|^2$ , $P_a$ is the orthogonal projection
from $\C^n$ onto the one dimensional subspace $[a]$ generated by $a$,
and $Q_a$ is the orthogonal projection from $\C^n$ onto the
orthogonal complement of $[a]$. When $a=0$, $\p_a(z)=-z$.
$\p_a$ has the following properties: $\p_a\circ\p_a(z)=z$,
and
$$
1-|\p_a(z)|^2=\frac{(1-|a|^2)(1-|z|^2)}{|1-\langle z,a\rangle|^2}.
$$
For $z,w\in\Bn$, the \emph{pseudo-hyperbolic distance} between $z$ and $w$
is defined by
$$
\rho(z,w)=|\p_z(w)|,
$$
and the \emph{hyperbolic distance} on $\Bn$ between $z$ and $w$
induced by the Bergman metric is given by
$$
\beta(z,w)=\tanh\,\rho(z,w).
$$
For $z\in\Bn$ and $r>0$, the \emph{Bergman metric ball} at $z$ is given by
$$
D(z,r)=\{w\in\Bn:\,\beta(z,w)<r\}.
$$
It is known that, for a fixed $r>0$, the weighted volume
$$
v_{\a}(D(z,r))\asymp (1-|z|^2)^{n+1+\a}.
$$
We refer to \cite{zhu2} for all of the above facts.

A sequence $\{a_k\}$ of points in $\Bn$ is a \emph{separated sequence} (in Bergman metric)
if there exists a positive constant $\delta>0$ such that
$\beta(z_i,z_j)>\delta$ for any $i\neq j$.
We need a well-known result on decomposition of the unit ball $\Bn$.
The following version is Theorem 2.23 in \cite{zhu2}

\begin{otherl}\label{covering}
There exists a positive integer $N$ such that for any $0<r<1$ we can
find a sequence $\{a_k\}$ in $\Bn$ with the following properties:
\begin{itemize}
\item[(i)] $\Bn=\cup_{k}D(a_k,r)$.
\item[(ii)] The sets $D(a_k,r/4)$ are mutually disjoint.
\item[(iii)] Each point $z\in\Bn$ belongs to at most $N$ of the sets $D(a_k,4r)$.
\end{itemize}
\end{otherl}

Any sequence $\{a_k\}$ satisfying the conditions of the above lemma is called a \emph{lattice} (or an $r$-$lattice$ if one wants to stress the dependence on $r$) in the Bergman metric. Obviously any $r$-lattice is separated.

For convenience, we will denote by $D_k=D(a_k,r)$ and $\tilde{D}_k=D(a_k,4r)$.
Then Lemma~\ref{covering} says that $\Bn=\cup_{k=1}^{\infty}D_k$
and there is an positive integer $N$ such that every point $z$ in $\Bn$
belongs to at most $N$ of sets $\tilde{D}_k$.

We also need the following atomic decomposition theorem for
weighted Bergman spaces. This turns out to be a powerful theorem in the theory of Bergman spaces.
The result is basically due to Coifman and Rochberg \cite{cr},
and can be found in Chapter 2 of \cite{zhu2}.
\begin{otherth}\label{AtomicD}
Suppose $p>0$, $\a>-1$, and
$$
b>n\max\left(1,\frac1p\right)+\frac{1+\a}{p}.
$$
Then we have
\begin{itemize}
\item[(i)] For any separated sequence $\{a_k\}$ in $\Bn$
and any sequence $\lambda=\{\lambda_k\}\in \ell^p$, the function
$$
f(z)
=\sum_{k=1}^{\infty} \lambda_k \frac{(1-|a_k|^2)^{b-(n+1+\a)/p}}{(1-\langle z,a_ k\rangle)^{b}}
$$
belongs to $A^p_{\a}$ and
$$
\|f\|_{p,\a}\lesssim\|\lambda\|_{\ell^p}.
$$
\item[(ii)] There is an $r$-lattice $\{a_ k\}$ in $\Bn$ such that, for any $f\in A^p_{\a}$, there is a sequence $\lambda=\{\lambda_k\}\in \ell^p$
with
$$
f(z)=\sum_{k=1}^{\infty} \lambda_k \frac{(1-|a_k|^2)^{b-(n+1+\a)/p}}{(1-\langle z,a_ k\rangle)^{b}}.
$$
and
$$
\|\lambda\|_{\ell^p}\lesssim \|f\|_{p,\a}.
$$
\end{itemize}
\end{otherth}

In the proof given in \cite{zhu2}, part (i) requires that
the sequence $\{a_k\}$ is an $r$-lattice for some $r\in(0,1]$, but it is well known that only the separation of the sequence $\{a_ k\}$ is needed.

\section{Weak factorizations and Hankel forms}
\label{hankel forms}

It is well known to specialists, but difficult to find in the literature,  that the obtention of weak factorizations is equivalent to estimates for small Hankel operators or Hankel forms (in our case, estimates with loss).
The equivalence between boundedness of the bilinear Hankel form $T_b^{\a}$ and
weak factorization can be formulated as the following result. Since this is the basis for our obtention of the weak factorization for Bergman spaces, for completeness,  we offer the proof here of the implication that gives the factorization.

\begin{prop}\label{equiv}
Let $1<q<\infty$ and $\a,\beta>-1$.
Let $p_1,p_2$ and $\a_1,\a_2$ satisfy (\ref{pa-ineq}) and (\ref{conjugate}),
and let $q'$ and $\beta'$ satisfy (\ref{conjugate-q}).
The following are equivalent:
\begin{itemize}
\item[(i)]
$A^q_{\beta}\subset A^{p_1}_{\a_1}\odot A^{p_2}_{\a_2}\,$ with $\|f\|_{A^{p_1}_{\a_1}\odot A^{p_2}_{\a_2}}\lesssim \|f\|_{q,\beta}\,$ for $f\in A^q_{\beta}$.
\item[(ii)] For any analytic function $b$, if
$T_b^{\a}$ is bounded on $A^{p_1}_{\a_1}\times A^{p_2}_{\a_2}$, then
 $b\in A^{q'}_{\beta'}$ with $\|b\|_{q',\beta'}\lesssim \|T_ b^{\a}\|$.
\end{itemize}
\end{prop}

\begin{proof}
We will prove (ii) implies (i). The other implication is easier, and the interested reader can follow the argument in  \cite[Corollary 1.2]{arsw} for a proof. By the atomic decomposition in Theorem \ref{AtomicD}, we have the inclusion $A^q_{\beta}\subset A^{p_1}_{\a_1}\odot A^{p_2}_{\a_2}.$ In order to have the corresponding estimate for the norms, we will show that, for any bounded linear functional $F$ on $A^{p_1}_{\a_1}\odot A^{p_2}_{\a_2}$, there is a unique function $b_ F\in A^{q'}_{\beta'}$ with $\|b_ F\|_{q',\beta'}\lesssim \|F\|$ such that $F(f)=\langle f,b_ F \rangle _{\alpha}$ for $f\in A^q_{\beta}$. This would give
\[
\|f\|_{A^{p_1}_{\a_1}\odot A^{p_2}_{\a_2}} =\sup_{\|F\|=1} |F(f)| \le \sup_{\|F\|=1} \|b_ F\|_{q',\beta'}\cdot \|f\|_{q,\beta} \lesssim \|f\|_{q,\beta}.
\]
Thus, suppose $F\in (A^{p_1}_{\a_1}\odot A^{p_2}_{\a_2})^*$
with norm $\|F\|$.
Then for all $\p\in A^{p_2}_{\a_2}$ we have
$$
|F(\p)|=|F(1\cdot\p)|\le \|F\|\cdot\|1\|_{p_1,\a_1}\cdot\|\p\|_{p_2,\a_2}
=\|F\|\cdot\|\p\|_{p_2,\a_2}.
$$
Hence $F\in (A^{p_2}_{\a_2})^*$, and so, by Theorem \ref{duality},
there is an unique function $b=b_ F\in A^{p_2'}_{\a_2'}$ such that
$F(\p)=\langle \p,b\rangle_{\a}$ for all $\p\in A^{p_2}_{\a_2}$,
where $p_2'$ and $\a_2'$ satisfy
$$
\frac{1}{p_2}+\frac{1}{p_2'}=1,\qquad
\frac{\a_2}{p_2}+\frac{\a_2'}{p_2'}=\a.
$$
Now, for polynomials $g$ and $h$ we have
\begin{displaymath}
\begin{split}
|T_b^{\a}(g,h)|&=|\langle gh,b\rangle_{\a}|
=|F(gh)|
\\
&
\le \|F\|\cdot\|g h\|_{A^{p_1}_{\a_1}\odot A^{p_2}_{\a_2}}
\\
&\le \|F\|\cdot\|g\|_{p_1,\a_1}\cdot\|h\|_{p_2,\a_2},
\end{split}
\end{displaymath}
which shows that $T_b^{\a}$ is bounded
on $A^{p_1}_{\a_1}\times A^{p_2}_{\a_2}$ with $\|T_b^{\a}\|\le \|F\|$. Therefore,
 we know that $b\in A^{q'}_{\beta '}$ with $\|b\|_{q',\beta'}\lesssim \|T_ b^{\a}\|\lesssim \|F\|$. Hence $\Lambda(f)=\langle f,b \rangle_{\alpha}$ defines a bounded linear form on $A^q_\beta$, and coincides with $F$ on polynomials. Thus, for $f\in A^q_{\beta}$, we have $F(f)=\Lambda(f)=\langle f,b \rangle _{\alpha}$.
The proof is complete.
\end{proof}

\section{Proof of Theorem~\ref{hankel}}
\label{proof-thm1}

In this section we prove Theorem~\ref{hankel}, from which
Theorem~\ref{hankel-form} follows. Before that,
for  $s\ge 0$ and $\a>-1$,
let  $R^{\a,s}$ denote the unique continuous linear operator on $H(\Bn)$ satisfying
\begin{displaymath}
R^{\alpha,s}\left (\frac{1}{(1-\langle z,w\rangle )^{n+1+\alpha}}\right )=\frac{1}{(1-\langle z,w\rangle )^{n+1+\alpha+s}}
\end{displaymath}
for all $w\in \Bn$. If $f\in A^1_{\alpha}$, then $R^{\alpha,s}f$ is given by the following integral expression
\begin{equation}\label{def-q}
R^{\a,s}f(z)=\int_{\Bn}\frac{f(w)}{(1-\langle z,w\rangle)^{n+1+\a+s}}\,dv_{\a}(w).
\end{equation}
More properties of the ``differential type" operators $R^{\alpha,s}$ can be found in \cite[Section 1.4]{zhu2}. Now we are ready to prove Theorem~\ref{hankel}.

\begin{proof}[\textbf{Proof of Theorem~\ref{hankel}}]
As we noticed before, we just need to prove that
$S^{\a}_f:\, A^{p_1}_{\a_1}\to A^{p_2}_{\a_2}$ is bounded
if and only if $f\in A^q_{\beta}$.
\par Suppose first that $f\in A^q_{\beta}$. We need to show $S^{\a}_f:A^{p_1}_{\a_1}\rightarrow A^{p_2}_{\a_2}$ is bounded. Let $g\in A^{p_1}_{\a_1}$. If $p_ 2 > 1$ then $P_{\alpha}:L^{p_ 2}(\Bn,dv_{\alpha_ 2})\rightarrow A^{p_ 2}_{\a_ 2}$ is bounded, and then from H\"{o}lder's inequality the result follows. Indeed,
\begin{displaymath}
\|S^{\a}_ f g\|_{p_ 2,\a_ 2}=\|P_{\a}(f\bar{g})\|_{p_ 2,\a_ 2} \le C \|fg\|_{p_ 2,\a_ 2}\le C \|f\|_{q,{\beta}}\cdot \|g\|_{p_ 1,\a_ 1}
\end{displaymath}
which shows that $S^{\a}_f:\, A^{p_1}_{\a_1}\to A^{p_2}_{\a_2}$ is bounded with
$$\|S^{\a}_f\|\lesssim \|f\|_{q,{\beta}}.$$

Conversely, suppose $S^{\a}_f:\, A^{p_1}_{\a_1}\to A^{p_2}_{\a_2}$ is bounded,
we are going to show that $f\in A^q_{\beta}$.
We begin with using an argument of Luecking (see, e.g., \cite{l2}).
Let $r_k(t)$ be a sequence of Rademacher functions
(see \cite[Appendix A]{Du}).
Let $b$ be large enough so that
\begin{equation}\label{b1}
b>n+\frac{1+\a_1}{p_1}.
\end{equation}
Fix any $r>0$, and let $\{a_k\}$ be an $r$-lattice and $\{D_k\}$
be the associated sets in Lemma~\ref{covering}.
By Theorem \ref{AtomicD}, we know that,
for any sequence of real numbers $\lambda=\{\lambda_k\}\in \ell^{p_1}$, the function
$$
g_t(z)=\sum_{k=1}^{\infty}\lambda_k r_k(t)
\frac{(1-|a_k|^2)^{b-(n+1+\alpha_1)/p_1}}{(1-\langle z, a_k\rangle)^b}
$$
belongs to $A^{p_1}_{\alpha_1}$ with $\|g_t\|_{p_1,\alpha_1}\lesssim \|\lambda\|_{\ell^{p_1}}$ for almost every $t$ in $(0,1)$.
Denote by
$$
g_k(z)=\frac{(1-|a_k|^2)^{b-(n+1+\alpha_1)/p_1}}{(1-\langle z,a_k\rangle)^b}.
$$
Since $S^{\a}_f:\, A^{p_1}_{\a_1}\to A^{p_2}_{\a_2}$ is bounded,
we get that
\begin{eqnarray*}
\|S^{\a}_fg_t\|_{p_2,\alpha_2}^{p_2}
&=&\int_{\Bn}\left|\sum_{k=1}^{\infty}\lambda_kr_k(t)S^{\a}_fg_k(z)\right|^{p_2}\,dv_{\alpha_2}(z)\\
\\
&\lesssim& \|S^{\a}_f\|^{p_2}\cdot \|g_t\|_{p_1,\alpha_1}^{p_2}
\lesssim \|S^{\a}_f\|^{p_2} \cdot \|\lambda\|_{\ell^{p_1}}^{p_2}
\end{eqnarray*}
for almost every $t$ in $(0,1)$.
Integrating both sides with respect to $t$ from $0$ to $1$,
and using Fubini's Theorem and Khinchine's inequality (see \cite[p.12]{wo}),
we get
\begin{equation}\label{lp}
\int_{\Bn}\left(\sum_{k=1}^{\infty}|\lambda_k|^2
|S^{\a}_fg_k(z)|^2\right)^{p_2/2}\!\!\!\!\!\!dv_{\alpha_2}(z)
\lesssim \|S^{\a}_f\|^{p_2} \cdot \|\lambda\|_{\ell^{p_1}}^{p_2}.
\end{equation}
Now we estimate
\begin{equation}\label{sf0}
\sum_{k=1}^{\infty}\!|\lambda_k|^{p_2}\!\!
\int_{\!\tilde D_k} \!\! |S^{\a}_fg_k(z)|^{p_2} dv_{\alpha_2}(z)
=\!\!\int_{\Bn}\!\!\!\left(\sum_{k=1}^{\infty}\!|\lambda_k|^{p_2}
|S^{\a}_fg_k(z)|^{p_2}\chi_{\tilde D_k}(z)\!\right)^{\frac{2}{p_2}\cdot\frac{p_2}{2}}
\!\!\!\!\!\!\!\!\!\!dv_{\alpha_2}(z).
\end{equation}
If $p_2\ge 2$, then $2/p_2\le 1$, and from (\ref{sf0}) we have
\begin{eqnarray*}
&~&\sum_{k=1}^{\infty}|\lambda_k|^{p_2}
\int_{\tilde D_k}|S^{\a}_fg_k(z)|^{p_2}\,dv_{\alpha_2}(z)\\
&~&\qquad
 \le \int_{\Bn}\left(\sum_{k=1}^{\infty}|\lambda_k|^{2}
|S^{\a}_fg_k(z)|^{2}\chi_{\tilde D_k}(z)\right)^{p_2/2}\!\!\!\!\! dv_{\alpha_2}(z)\\
&~&\qquad \le \int_{\Bn}\left(\sum_{k=1}^{\infty}|\lambda_k|^{2}
|S^{\a}_fg_k(z)|^{2}\right)^{p_2/2}
\!\!\!\!\!dv_{\alpha_2}(z).
\end{eqnarray*}
If $1<p_2<2$, then $2/p_2>1$, from (\ref{sf0}), by H\"older's inequality we get
\begin{eqnarray*}
&~&\sum_{k=1}^{\infty}|\lambda_k|^{p_2}
\int_{\tilde D_k}|S^{\a}_fg_k(z)|^{p_2}\,dv_{\alpha_2}(z)\\
&~&\qquad \le
\int_{\Bn}\left(\sum_{k=1}^{\infty}|\lambda_k|^2
|S^{\a}_fg_k(z)|^2\right)^{p_2/2}
\left(\sum_{k=1}^{\infty}\chi_{\tilde D_k}(z)\right)^{1-p_2/2} \!\!\!\! dv_{\alpha_2}(z)\\
&~&\qquad \le N^{1-p_2/2}
\int_{\Bn}\left(\sum_{k=1}^{\infty}|\lambda_k|^2
|S^{\a}_fg_k(z)|^2\right)^{p_2/2} \!\!\!\! dv_{\alpha_2}(z),
\end{eqnarray*}
since each $z\in\Bn$ belongs to at most $N$ of the sets $\tilde D_k$.
Combining the above two inequalities, and applying (\ref{lp}) we have
\begin{eqnarray*}
&~&\sum_{k=1}^{\infty}|\lambda_k|^{p_2}
\int_{\tilde D_k}|S^{\a}_fg_k(z)|^{p_2}\,dv_{\alpha_2}(z)\\
&~&\qquad\le\min\{1,N^{1-p_2/2}\}
\int_{\Bn}\left(\sum_{k=1}^{\infty}|\lambda_k|^2
|S^{\a}_fg_k(z)|^2\right)^{p_2/2} \!\!\!\! dv_{\alpha_2}(z)\\
&~&\qquad\lesssim \|S^{\a}_f\|^{p_2} \cdot \|\lambda\|_{\ell^{p_1}}^{p_2}.
\end{eqnarray*}
By subharmonicity we know that,
$$
|S^{\a}_fg_k(a_k)|^{p_2}
\lesssim\frac{1}{(1-|a_k|^2)^{n+1+\alpha_2}}
\int_{\tilde D_k}|S^{\a}_fg_k(z)|^{p_2}\,dv_{\alpha_2}(z).
$$
From this we obtain
\begin{equation}\label{sf1}
\sum_{k=1}^{\infty}|\lambda_k|^{p_2}(1-|a_k|^2)^{n+1+\alpha_2}\,
\big |S^{\a}_fg_k(a_k)\big |^{p_2}
\lesssim \|S^{\a}_f\|^{p_2}\cdot \|\lambda\|_{\ell^{p_1}}^{p_2}.
\end{equation}
Let
$
R^{\a,b}$ be the integral operator defined in \eqref{def-q}.
Then
\begin{eqnarray*}
S^{\a}_fg_k(a_k)
&=&\int_{\Bn}\frac{f(w)\overline{g_k(w)}}{(1-\langle a_k,w\rangle)^{n+1+\a}}\,dv_{\a}(w)\\
&=&\int_{\Bn}\frac{f(w)(1-|a_k|^2)^{b-(n+1+\a_1)/p_1}}
{(1-\langle a_k,w\rangle)^{n+1+\a}(1-\langle a_k,w\rangle)^{b}}\,dv_{\a}(w)\\
&=&(1-|a_k|^2)^{b-(n+1+\a_1)/p_1}
\int_{\Bn}\frac{f(w)}{(1-\langle a_k,w\rangle)^{n+1+\a+b}}\,dv_{\a}(w)\\
&=&(1-|a_k|^2)^{b-(n+1+\a_1)/p_1}R^{\a,b}f(a_k).
\end{eqnarray*}
Thus (\ref{sf1}) becomes
\begin{equation}\label{sf2}
\sum_{k=1}^{\infty}\!|\lambda_k|^{p_2}(1-|a_k|^2)^{(n+1+\alpha_2)+[b-(n+1+\a_1)/p_1]p_2}
|R^{\a,b} \!f(a_k)|^{p_2}
\lesssim \|S^{\a}_f\|^{p_2}\cdot \|\lambda \|_{\ell^{p_1}}^{p_2}.
\end{equation}
Since
$$
(n+1+\alpha_2)+\left(b-\frac{n+1+\a_1}{p_1}\right)p_2
=\left(b+\frac{n+1+\beta}{q}\right)p_2,
$$
the equation (\ref{sf2}) is the same as
\begin{equation}\label{sf3}
\sum_{k=1}^{\infty}|\lambda_k|^{p_2}
\left[(1-|a_k|^2)^{b+(n+1+\beta)/q}|R^{\a,b} \!f(a_k)|\right]^{p_2}
\lesssim \|S^{\a}_f\|^{p_2}\cdot \|\lambda\|_{\ell^{p_1}}^{p_2}.
\end{equation}
Since $\{\lambda_k\}$ was an arbitrary sequence in $\ell^{p_1}$,
we know that $\{\lambda_k^{p_2}\}$ is an arbitrary sequence in $\ell^{p_1/p_2}$.
Since the conjugate exponent of $p_ 1/p_ 2$ is $(p_1/p_2)'=p_1/(p_1-p_2)$, by duality we obtain that
$$
\left\{(1-|a_k|^2)^{b+(n+1+\beta)/q}|R^{\a,b} f(a_k)|\right\}\in \ell^{p_1p_2/(p_1-p_2)}=\ell^q,
$$
and
\begin{equation}\label{sf4}
\sum_{k=1}^{\infty}(1-|a_k|^2)^{bq+(n+1+\beta)}|R^{\a,b} \!f(a_k)|^q
\lesssim \|S^{\a}_f\|^q.
\end{equation}
This is the discrete version of what we want. Now, we will deduce that $f\in A^q_{\beta}$ with $\|f\|_{q,\beta}\lesssim \|S^{\a}_f\|$ using
 duality and the atomic decomposition for Bergman spaces. Indeed,
choose $\beta'=q'(\a-\beta/q)$ (which means $\a=\beta/q+\beta'/q'$).
 Note that condition (\ref{alpha}) implies that $1+\a>(1+\beta)/q$, and this
guarantees that $\beta'>-1$. Hence, by the duality result in Theorem \ref{duality},
\begin{equation}\label{Eq-D9}
\|f\|_{q,\beta}
\asymp\sup_{\|h\|_{q'\!,\beta'}=1}|\langle h,f \rangle_{\a}|.
\end{equation}
Observe that
$$
n+1+\a+b>n+\frac{1+\beta'}{q'}.
$$
Then, we can apply \eqref{sf4} with the $r$-lattice $\{a_ k\}$ for which the atomic decomposition in Theorem \ref{AtomicD} for $A^{q'}_{\beta'}$ holds.
That is, for any $h\in A^{q'}_{\beta'}$, there exists a sequence $\mu=\{\mu_k\}\in \ell^{q'}$
with $\|\mu\|_{\ell^{q'}}\lesssim\|h\|_{q',\beta'}$ such that
$$
h(z)=\sum_{k=1}^{\infty}\mu_k
\frac{(1-|a_k|^2)^{n+1+\a+b-(n+1+\beta')/q'}}{(1-\langle z, a_k\rangle)^{n+1+\a+b}}.
$$
Since
$$
n+1+\a+b-(n+1+\beta')/q'
=b+(n+1+\beta)/q,
$$
then \eqref{Eq-D9}, H\"older's inequality and \eqref{sf4} gives
\begin{eqnarray*}
\|f\|_{q,\beta}
&\asymp&\sup_{\|h\|_{q'\!,\beta'}=1}\left| \, \sum_{k=1}^{\infty}\mu_k
(1-|a_k|^2)^{b+(n+1+\beta)/q}\,\overline{R^{\a,b} f(a_k)}\,\right|
\\
&\le&\sup_{\|h\|_{q'\!,\beta'}=1} \big \|\mu \big \|_{\ell^{q'}}
\left[\sum_{k=1}^{\infty}(1-|a_k|^2)^{bq+(n+1+\beta)}|R^{\a,b} f(a_k)|^q\right]^{1/q}
\\
&\lesssim & \|S^{\a}_f\|.
\end{eqnarray*}
Hence $f\in A^q_{\beta}$ with $$
\|f\|_{q,\beta}\lesssim \|S^{\a}_f\|.
$$
This finishes the proof.

\end{proof}

\section{Further results}\label{Com}
\subsection{Compactness}
 Under the assumptions of Theorem \ref{hankel}, actually one has that the small Hankel operator $h^{\a}_{\overline{f}}:A^{p_ 1}_{\a_ 1}\rightarrow \overline{A^{p_ 2}_{\a_ 2}}$ is bounded if and only if it is compact. This is from a general result of Banach space theory.
It is known that, for $0< p_2<p_1<\infty$, every bounded operator
from $\ell^{p_1}$ to $\ell^{p_2}$ is compact
(see, for example Theorem I.2.7, p.31 in \cite{lt}).
Since the weighted Bergman space $A^p_{\a}$ is isomorphic to $\ell^{p}$
(see, Theorem 11, p.89 in \cite{wo}, note that the same proof there works
for weighted Bergman spaces on the unit ball $\Bn$),
we get directly the above result.

\subsection{Small Hankel operators with the same weights}

Concerning the boundedness of the small Hankel operator $h^{\alpha}_{\overline{f}}:A^{p_ 1}_{\a}\rightarrow \overline{A^{p_ 2}_{\a}}$ (the case when all the weights are the same) for all possible choices of $0<p_ 1,p_ 2<\infty$, we mention here that the case $p_ 1=p_ 2>1$ is by now classical (see \cite{JPR}, \cite{Zhu-TAMS} and \cite{bl}), and in this case the boundedness is equivalent to the symbol $f$ being in the Bloch space $\mathcal{B}$, that consists of those holomorphic functions $f$ on $\Bn$ with
$$\|f\|_{\mathcal{B}}=|f(0)|+\sup_{z\in \Bn} (1-|z|^2) |Rf(z)|<\infty.$$
Here, $Rf$ denotes the radial derivative of $f$, that is, $$Rf(z)= \sum_{k=1}^{n} z_ k \frac{\partial f}{\partial z_ k} (z),\qquad z=(z_ 1,\dots,z_ n)\in \Bn.$$
The Bloch space also admits an equivalent norm in terms of the invariant gradient $\widetilde{\nabla} f(z):=\nabla (f\circ \varphi_ z)(0)$ as follows
$$\|f\|_{\mathcal{B}}\asymp |f(0)|+\sup_{z\in \Bn} |\widetilde{\nabla} f(z)|.$$
The case $0<p_ 1\le p_ 2$ is completely settled in \cite{bl} (actually the results are stated for the unweighted Bergman spaces $A^p$, but the proofs works also for the weighted case). The description for the case $p_ 1=p_ 2=1$ is that $f$ must belong to the so called logarithmic Bloch space, a result that goes back to the one dimensional result of Attele \cite{At}. Concerning estimates with loss, in \cite{bl} Bonami and Luo obtained a description for the case $0<p_ 2<p_ 1$ with $p_ 2<1$ (again the result in \cite{bl} is stated for the unweighted Bergman spaces). Thus, in view of Corollary \ref{C4}, to complete the picture it remains to deal with the case $p_1>p_ 2=1$ (this problem is also open for the unit disk). In that case, also in \cite{bl}, some partial results are obtained (again for the unweighted case). Mainly, they provide a pointwise estimate that is necessary for the small Hankel operator to be bounded, and they show that the condition
\begin{equation}\label{C-BLZ}
f(z) \log \frac{2}{1-|z|} \in L^{p'_ 1}(\Bn,dv_{\alpha})
\end{equation}
is sufficient. Moreover, they conjecture that the previous condition is also necessary. We have not been able to prove the conjecture, but we are going to shed some light on that problem.

\begin{thm}\label{Tm4}
Let $f\in H(\Bn)$, $\alpha>-1$ and $p_ 1>1$. Let $p'_ 1$ be the conjugate exponent of $p_ 1$. Then $h^{\a}_ {\overline{f}}:A^{p_ 1}_{\a}\rightarrow \overline{A^{1}_{\a}}$ is bounded if and only if the multiplication operator $M_ f:\mathcal{B}\rightarrow A^{p'_ 1}_{\a}$ is bounded.
\end{thm}

Before going to the proof we need first some preparation. First of all, recall that the Bloch space is the dual of $A^{1}_{\a}$ under the integral pairing $\langle \,,\,\rangle _{\alpha}$ (see \cite[Theorem 3.17]{zhu2}). We also need the following lemma, whose one dimensional analogue is essentially proved in \cite{BP}.

\begin{lemma}\label{LBP}
Let $1<p<\infty$, $\sigma>-1$, and $n+1+\sigma<b$.
Then
\begin{displaymath}
\int_{\Bn}\frac{|f(z)-f(a)|^p}{|1-\langle a,z\rangle|^{b}}\,dv_{\sigma}(z)\lesssim
\int_{\Bn} |\widetilde{\nabla}f(z)|^p\,\frac{dv_{\sigma}(z)}{|1-\langle a,z\rangle|^{b}}
\end{displaymath}
for any $f\in H(\Bn)$ and $a\in \Bn$.
\end{lemma}

\begin{proof}
We are going to prove first that, for $0\le t<n+1+\sigma$,
\begin{equation}\label{EqLem-0}
\int_{\Bn}\frac{|f(z)-f(0)|^p}{|1-\langle a,z\rangle |^t}\,dv_{\sigma}(z)\lesssim \int_{\Bn}\!\!\frac{(1-|w|^2)^p\,|Rf(w)|^p }{|1-\langle a,w\rangle |^t}\,dv_{\sigma}(w).
\end{equation}
From \cite[p.51]{zhu2}, for $\beta$ big enough, say $\beta\ge 1+\sigma$, we have
$$|f(z)-f(0)|\le C \int_{\Bn} \frac{(1-|w|^2)\,|Rf(w)|\,dv_{\beta-1}(w)}{|1-\langle z,w \rangle |^{n+\beta}}.$$
Take a small number $\e>0$ with $\sigma-\e \max(p,p')>-1$, where $p'$ denotes the conjugate exponent of $p$, and $t<n+1+\sigma-\varepsilon p$. An application of H\"{o}lder's inequality and Lemma \ref{LA} yields
\begin{equation*}
|f(z)-f(0)|^p\lesssim (1-|z|^2)^{-\varepsilon p}\int_{\Bn} \frac{(1-|w|^2)^p\,|Rf(w)|^p\,dv_{\beta-1+\varepsilon p}(w)}{|1-\langle z,w \rangle |^{n+\beta}}.
\end{equation*}
This together with Fubini's theorem and \cite[Lemma 2.5]{OF} gives
\begin{equation*}
\begin{split}
\int_{\Bn}&\frac{|f(z)-f(0)|^p}{|1-\langle a,z\rangle |^t}\,dv_{\sigma}(z)
\\
&\lesssim \int_{\Bn}\!\!(1-|w|^2)^p\,|Rf(w)|^p \left(\int_{\Bn} \!\frac{dv_{\sigma-\varepsilon p}(z)}{|1-\langle a,z\rangle |^t\,|1-\langle z,w \rangle |^{n+\beta}}\right )dv_{\beta-1+\varepsilon p}(w)
\\
&\lesssim \int_{\Bn}\!\!\frac{(1-|w|^2)^p\,|Rf(w)|^p }{|1-\langle a,w\rangle |^t}\,dv_{\sigma}(w)
\end{split}
\end{equation*}
proving \eqref{EqLem-0}. Now, a change of variables $z=\varphi_ a(\zeta)$ gives (see \cite[Proposition 1.13]{zhu2})

\begin{displaymath}
\int_{\Bn}\!\!\!\frac{|f(z)\!-\!f(a)|^p}{|1-\langle a,z\rangle|^{b}}\,dv_{\sigma}(z)=\!\int_{\Bn}\!\!\!\frac{|(f\circ \varphi_ a)(\zeta)\!-\!(f\circ \varphi_ a)(0)|^p}{|1-\!\langle a,\varphi_ a(\zeta)\rangle|^{b}}\frac{(1-\!|a|^2)^{n+1+\sigma}}{|1\!-\!\langle a,\zeta \rangle |^{2(n+1+\sigma)}}\,dv_{\sigma}(\zeta).
\end{displaymath}
\mbox{}
\\
From \cite[Lemma 1.3]{zhu2} we have
\begin{displaymath}
1-\langle a, \varphi_ a(\zeta) \rangle =1-\langle \varphi_ a(0), \varphi_ a(\zeta) \rangle=\frac{1-|a|^2}{1-\langle a,\zeta \rangle }.
\end{displaymath}
Therefore we obtain

\begin{displaymath}
\int_{\Bn}\!\!\!\frac{|f(z)\!-\!f(a)|^p}{|1-\langle a,z\rangle|^{b}}\,dv_{\sigma}(z)=(1-\!|a|^2)^{n+1+\sigma-b}\!\int_{\Bn}\!\!\!\frac{|(f\circ \varphi_ a)(\zeta)\!-\!(f\circ \varphi_ a)(0)|^p}{|1-\!\langle a,\zeta\rangle|^{2(n+1+\sigma)-b}}\,dv_{\sigma}(\zeta).
\end{displaymath}
Due to our condition $b>n+1+\sigma$, we have $$t=2(n+1+\sigma)-b<n+1+\sigma$$ and we can apply \eqref{EqLem-0} to get

\begin{displaymath}
\int_{\Bn}\!\!\frac{|f(z)\!-\!f(a)|^p}{|1-\langle a,z\rangle|^{b}}\,dv_{\sigma}(z)\lesssim (1-\!|a|^2)^{n+1+\sigma-b}\!\int_{\Bn}\!\!\!\frac{(1-|\zeta|^2)^p\,|R(f\circ \varphi_ a)(\zeta)|^p}{|1-\langle a,\zeta\rangle|^{2(n+1+\sigma)-b}}\,dv_{\sigma}(\zeta).
\end{displaymath}
\mbox{}
\\
Since
\[
(1-|\zeta|^2)\,|R(f\circ \varphi_ a)(\zeta)|\le |\widetilde{\nabla}(f\circ \varphi_ a)(\zeta)|=|\widetilde{\nabla}f( \varphi_ a(\zeta))|,
\]
another change of variables $w=\varphi_ a(\zeta)$ finally gives

\begin{displaymath}
\int_{\Bn}\!\frac{|f(z)\!-\!f(a)|^p}{|1-\langle a,z\rangle|^{b}}\,dv_{\sigma}(z)\lesssim \int_{\Bn}\!\frac{|\widetilde{\nabla}f(w)|^p}{|1-\langle a,w\rangle|^{b}}\,dv_{\sigma}(w).
\end{displaymath}
completing the proof of the lemma.
\end{proof}
\mbox{}
\\
After these preparations, we are now ready for the proof of Theorem \ref{Tm4}.\\

\begin{proof}[\textbf{Proof of Theorem \ref{Tm4}}]
Assume first that the small Hankel operator $h^{\a}_ {\overline{f}}:A^{p_ 1}_{\a}\rightarrow \overline{A^{1}_{\a}}$ is bounded.
Let $g\in A^{p_ 1}_{\alpha}$. From the pointwise estimate for Bergman spaces, we get
$$|\langle g,f\rangle_{\alpha}|=|h_{\overline{f}}^{\alpha} g(0)|\le C \|h^{\alpha}_{\overline{f}} g \|_{1,\alpha}\le C \|h_{\overline{f}}^{\a}\|\cdot \|g\|_{p_ 1,\a}.$$
Therefore, by duality, we have that $f\in A^{p'_ 1}_{\a}$ with
\begin{equation}\label{EqB-0}
\|f\|_{p'_ 1,\a}\le C \|h_{\overline{f}}^{\a}\|.
\end{equation}

Recall that $h^{\a}_ {\overline{f}}:A^{p_ 1}_{\a}\rightarrow \overline{A^{1}_{\a}}$ is bounded, if and only if, $S^{\a}_ {f}:A^{p_1}_{\a}\rightarrow A^{1}_{\a}$ is bounded, with $\|S^{\a}_ f\|\asymp \|h^{\a}_{\overline{f}}\|$.
Also, since for any $g\in A^{p_1}_{\a}$ and $h\in\mathcal{B}$,
\begin{equation}\label{sfg}
\langle S_f^{\a}g,h\rangle_{\a}=\langle f,gh\rangle_{\a}=\langle S_f^{\a}h,g\rangle_{\a}
\end{equation}
we know that  $S^{\a}_ {f}: \mathcal{B}\rightarrow A^{p_1'}_{\a}$ is bounded, and moreover,
we have
 $$\|S^{\a}_ f\|_{\mathcal{B}\rightarrow A^{p'_ 1}_{\a}}\lesssim  \|h^{\a}_{\overline{f}}\|.$$
For $g$ in the Bloch space $\mathcal{B}$, one has
\begin{equation}\label{EqB-1}
\begin{split}
\|M_ f g\|^{p'_ 1}_{p'_ 1,\a}&=\int_{\Bn} |f(z)\,\overline{g(z)}\,|^{p'_ 1}\,dv_{\a}(z)
\\
&\lesssim \int_{\Bn} |S^{\a}_f g(z)\,|^{p'_ 1}\,dv_{\a}(z)+\int_{\Bn} |f(z)\,\overline{g(z)}-S_ f^{\a} g(z)|^{p'_ 1}\,dv_{\a}(z).
\end{split}
\end{equation}
Due to the boundedness of  $S^{\a}_ {f}:\mathcal{B}\rightarrow A^{p'_ 1}_{\a}$,
\begin{equation}\label{EqB-2}
\int_{\Bn} |S^{\a}_f g(z)\,|^{p'_ 1}\,dv_{\a}(z)\le \|S^{\a}_ f\|^{p'_ 1}_{\mathcal{B}\rightarrow A^{p'_ 1}_{\a}}\cdot \|g\|_{\mathcal{B}}^{p'_ 1}
\lesssim \|h^{\a}_{\overline{f}}\|^{p'_ 1}\cdot \|g\|_{\mathcal{B}}^{p'_ 1}.
\end{equation}
On the other hand, by the reproducing formula for Bergman spaces and H\"{o}lder's inequality,
\begin{displaymath}
\begin{split}
 |f(z)\,&\overline{g(z)}-S_ f^{\a} g(z)|^{p'_ 1}=\left |\int_{\Bn} \frac{f(w)\,(\overline{g(z)-g(w)})}{(1-\langle z,w \rangle)^{n+1+\a}}\,dv_{\alpha}(w) \right |^{p'_ 1}
\\
& \le  \left ( \int_{\Bn}\frac{|f(w)|^{p'_ 1}}{|1-\langle w,z \rangle|^{n+1+\a}}\,dv_{\a+\e p'_ 1}(w) \right ) \left ( \int_{\Bn}\frac{|g(z)-g(w)|^{p_ 1}}{|1-\langle w,z \rangle|^{n+1+\a}}\,dv_{\a-\e p_ 1}(w) \right )^{\frac{p'_ 1}{p_ 1}},
\end{split}
\end{displaymath}
where $\e>0$ satisfies $\a-\e \max(p_ 1,p'_ 1)>-1$. Using Lemma \ref{LBP} and Lemma \ref{LA} we get
\begin{displaymath}
\begin{split}
\int_{\Bn}\frac{|g(z)-g(w)|^{p_ 1}}{|1-\langle w,z \rangle|^{n+1+\a}}\,dv_{\a-\e p_ 1}(w)&\lesssim \int_{\Bn}\frac{|\widetilde{\nabla} g(w)|^{p_ 1}}{|1-\langle w,z \rangle|^{n+1+\a}}\,dv_{\a-\e p_ 1}(w)\\
\\
&\lesssim \|g\|_{\mathcal{B}}^{p_ 1} (1-|z|^2)^{-\e p_ 1}.
\end{split}
\end{displaymath}
Therefore, this together with Fubini's theorem, Lemma \ref{LA} and the estimate \eqref{EqB-0} gives
\begin{equation}\label{EqB-3}
\begin{split}
 \int_{\Bn}|f(z)\,\overline{g(z)}&-S_ f^{\a} g(z)|^{p'_ 1}\,dv_{\a}(z)
 \\
 &
\le C  \|g\|_{\mathcal{B}}^{p'_ 1} \int_{\Bn} |f(w)|^{p'_ 1}\left ( \int_{\Bn}\frac{dv_{\a-\e p'_ 1}(z)}{|1-\langle w,z \rangle|^{n+1+\a}} \right )dv_{\a+\e p'_ 1}(w)
\\
&
\le C \|g\|_{\mathcal{B}}^{p'_ 1}\cdot \|f\|_{p'_ 1,\a}^{p'_ 1}\le C \|h_{\overline{f}}^{\a}\|^{p'_ 1}\cdot \|g\|_{\mathcal{B}}^{p'_ 1}.
\end{split}
\end{equation}
Putting together the estimates \eqref{EqB-1}, \eqref{EqB-2} and \eqref{EqB-3} it follows that $M_ f:\mathcal{B}\rightarrow A^{p'_ 1}_{\a}$ is bounded with $\|M_ f\|_{\mathcal{B}\rightarrow A^{p'_ 1}_{\a}}\lesssim \|h_{\overline{f}}^{\a}\|$.\\

Conversely, suppose that $M_ f:\mathcal{B}\rightarrow A^{p'_ 1}_{\a}$ is bounded. By the boundedness of the projection $P_{\a}:L^{p'_1}(\Bn,dv_{\a})\rightarrow A^{p'_ 1}_{\a}$ one deduces that $S_f ^{\a}:\mathcal{B}\rightarrow A^{p'_ 1}_{\a}$ is also bounded,
and so obviously, $S_f ^{\a}:\mathcal{B}_0\rightarrow A^{p'_ 1}_{\a}$ is bounded, where $\mathcal{B}_0$
is the little Bloch space, and it is well-known that
the dual space of $\mathcal{B}_0$ is $A^1_{\a}$ under the integral pair $\langle\,,\,\rangle_{\a}$
(see, for example, Chapter 3 of \cite{zhu2}),
from (\ref{sfg}) we know that $S_f^{\a}:A^{p_1}_{\a}\rightarrow A^{1}_{\a}$ is bounded.
\end{proof}

As a consequence of Theorem \ref{Tm4} we can easily obtain the sufficient and necessary conditions given in \cite{bl} as well as another relevant necessary condition for the boundedness of $h^{\a}_ {\overline{f}}:A^{p_ 1}_{\a}\rightarrow \overline{A^{1}_{\a}}$.
\begin{cor}
Let $f\in H(\Bn)$, $\alpha>-1$ and $p_ 1>1$.
\begin{enumerate}
\item[(i)] If \eqref{C-BLZ} holds, then $h^{\a}_ {\overline{f}}:A^{p_ 1}_{\a}\rightarrow \overline{A^{1}_{\a}}$ is bounded.

\item[(ii)] If $h^{\a}_ {\overline{f}}:A^{p_ 1}_{\a}\rightarrow \overline{A^{1}_{\a}}$ is bounded, then
\begin{equation}\label{Cor-1}
 \sup_{z\in \Bn} (1-|z|^2)^{(n+1+\a)/p'_ 1}\, |f(z)|\,\Big (\log \frac{2}{1-|z|^2} \Big)<\infty
\end{equation}
and
\begin{equation}\label{Cor-2}
\int_{\Bn} |f(z)|^{p'_ 1}\,\Big (\log \frac{2}{1-|z|^2} \Big)^{\frac{p'_ 1}{2}} \!dv_{\alpha}(z)<\infty.
\end{equation}
\end{enumerate}
\end{cor}
\begin{proof}
Part (i) follows directly from Theorem \ref{Tm4} and the pointwise estimate for Bloch functions
$$|g(z)|\le \|g\|_{\mathcal{B}} \,\log \frac{2}{1-|z|^2}.$$
To prove part (ii), for each $z\in \Bn$, the function $$g_ z(w)=\log \frac{2}{1-\langle w,z \rangle}$$ is in the Bloch space with $\|g_ z\|_{\mathcal{B}}\le C$ with the constant $C$ independent of the point $z$. Therefore, from the pointwise estimate for functions in Bergman spaces, we get
\begin{displaymath}
\begin{split}
(1-|z|^2)^{n+1+\a}\Big (|f(z)| \log \frac{2}{1-|z|^2}\Big )^{p'_ 1}&=(1-|z|^2)^{n+1+\a}|f(z)\,g_ z(z)|^{p'_ 1}
\\
&\lesssim \|fg_ z\|_{p'_ 1,\a}^{p'_ 1}
=\|M_ f g_ z\|_{p'_ 1,\a}^{p'_ 1}
\\
&\le \|M_ f\|_{\mathcal{B}\rightarrow A^{p'_ 1}_{\a}}\cdot \|g_ z\|_{\mathcal{B}}^{p'_ 1}
\lesssim \|M_ f\|_{\mathcal{B}\rightarrow A^{p'_ 1}_{\a}},
\end{split}
\end{displaymath}
and \eqref{Cor-1} follows due to Theorem \ref{Tm4}. The necessity of \eqref{Cor-2} is also a consequence of Theorem \ref{Tm4}. Indeed, clearly  $M_ f:\mathcal{B}\rightarrow A^{p'_ 1}_{\a}$ is bounded if and only if the measure $d\mu_ f(z)=|f(z)|^{p'_ 1}\,dv_{\alpha}(z)$ is a $p'_1$- Carleson measure for the Bloch space (see \cite{Doub,GPGR} for the definition); and by Proposition 1.4 in \cite{Doub} (the one dimensional case appears in \cite{GPGR} and \cite{Limp}) this implies \eqref{Cor-2}, finishing the proof.
\end{proof}
The established connection between Hankel operators on Bergman spaces and Carleson measures for the Bloch space makes even more interesting the problem (as far as we know, still open) of describing those measures.\\

\textbf{Acknowledgments:} The authors would like to thank the referee for valuable
comments that improved the final version of the paper.


\begin{thebibliography}{99}

\bibitem{arsw}
N. Arcozzi, R. Rochberg, E. Sawyer \and B. D. Wick,
{\it Bilinear forms on the Dirichlet space},
  Anal. PDE 3 (2010), 21--47.

\bibitem{At} K. Attele, {\it Toeplitz and Hankel operators on Bergman one space},  Hokkaido Math. J. 21 (1992), 279--293.

\bibitem{BP} D. Blasi \and J. Pau, {\it A characterization of Besov
type spaces and applications to Hankel type operators},  Michigan
Math. J. 56 (2008), 401-417.

\bibitem{bl} A. Bonami \and L. Luo,
{\it On Hankel operators between Bergman spaces on the unit ball},
 Houston J. Math. 31 (2005), 815--828.

\bibitem{cr} R. Coifman \and R. Rochberg,
{\it Representation theorems for holomorphic and harmonic functions in $L^p$},
 Asterisque  77 (1980), 11-66.

\bibitem{crw} R. Coifman, R. Rochberg \and G. Weiss,
{\it Factorization theorems for Hardy spaces in several variables},
 Ann. of Math. (2) 103 (1976), 611--635.

\bibitem{Doub} E. Doubtsov, {\it Carleson-Sobolev measures for weighted Bloch spaces},  J. Funct. Anal. 258 (2010), 2801--2816.

\bibitem{Du}
P.L. Duren, `Theory of $H\sp{p} $ Spaces', Academic Press,
New~York-London 1970. Reprint: Dover, Mineola, New York 2000.

\bibitem{GL} J. Garnett \and R. Latter, {\it The atomic decomposition for Hardy spaces in several complex variables},  Duke Math. J. 45 (1978), 815--845.

\bibitem{GPGR} D. Girela, J. A. Pel\'{a}ez, F. P\'{e}rez-Gonz\'{a}lez \and J.
R\"{a}tty\"{a}, {\it Carleson measures for the Bloch space},   Integral
Equations and Operator Theory 61 (2008), 511--547.


\bibitem{Gow} M. Gowda, {\it Nonfactorization theorems in weighted Bergman and Hardy spaces on the unit ball of $\Cn$},
 Trans. Amer. Math. Soc. 277 (1983), 203--212.

\bibitem{ho} C. Horowitz,
{\it Factorization theorems for functions in the Bergman spaces},
 Duke Math. J.,  44 (1977), 201--213.

\bibitem{JPR} S. Janson, J. Peetre \and R. Rochberg, {\it Hankel forms and the Fock
space},  Revista Mat. Iberoamericana 3 (1987), 61--138.

\bibitem{Limp} T.G. Limperis, `Embedding theorems for the Bloch space', PhD thesis, University of Arkansas, 1998.

\bibitem{lt} J. Lindenstrauss \and L. Tzafriri,
`Classical Banach Spaces',
Lecture Notes in Math. 338,
Springer-Verlag, Berlin, 1973.

\bibitem{l1} D. H. Luecking,
{\it Representations and duality in weighted spaces of analytic functions},
 Indiana Univ. Math. J.
34 (1985) 319--336.

\bibitem{l2} D. H. Luecking,
{\it Embedding theorems for spaces of analytic functions via Khinchine's inequality},
Michigan Math. J. 40 (1993), 333--358.


\bibitem{ne} Z. Nehari,
{\it On bounded bilinear forms},
 Ann. of Math. (2) 65 (1957), 153--162.

\bibitem{OF} J.M. Ortega \and J. Fabrega, {\it Pointwise multipliers and corona type decomposition in $BMOA$},  Ann. Inst. Fourier (Grenoble) 46 (1996), 111--137.

\bibitem{Roc} R. Rochberg, {\it Decomposition theorems for Bergman spaces and their applications}, in `Operators and Function Theory', D. Reidel, 1985, 225--277.

\bibitem{wo} P. Wojtaszczyk,
`Banach Spaces for Analysts',
Cambridge Studies in Advanced Mathematics 25.
Cambridge University Press, 1991.

\bibitem{zz} R. Zhao and K. Zhu,
`Theory of Bergman spaces in the unit ball of $\C^n$',
 Mem. Soc. Math. Fr. (N.S.), 115 (2008), vi+103 pp.

\bibitem{Zhu-TAMS} K. Zhu, {\it Hankel operators on the Bergman spaces of bounded symmetric domains}, Trans. Amer. Math. Soc. 324 (1991), 707--730.

\bibitem{zhu2} K. Zhu,
`Spaces of Holomorphic Functions in the Unit Ball',
Springer-Verlag, New York, 2005.




\end{thebibliography}
\end{document}